\documentclass[preprint,11pt]{elsarticle}

\usepackage{graphicx}
\usepackage{amssymb}

\usepackage{lineno}

\journal{Journal Name}
\usepackage{amsmath}
\usepackage{amssymb}
\usepackage{amsthm}
\usepackage{graphicx}
\usepackage{graphics}
\usepackage{color}
\newcommand{\eps}{\varepsilon}

\newcommand{\Real}{\mathbf{R}}

\newcommand{\chro}{\overrightarrow{\rm exp}\int}
\newcommand{\ad}{\mbox{\rm ad}}

\newtheorem{mthm}{Theorem}

\newtheorem{thm}{Theorem}[section]

\newtheorem{lem}[thm]{Lemma}
\newtheorem{prop}[thm]{Proposition}
\newtheorem{defn}[thm]{Definition}
\newtheorem{rem}[thm]{Remark}
\newtheorem{asmp}[thm]{Assumption}
\newtheorem{questi}{Question}
\begin{document}

\begin{frontmatter}


\title{Ensemble Controllability by Lie algebraic methods}
\author[SISSA,MIAN]{A.~Agrachev}
\address[SISSA]{Internat. School for Advanced Studies (SISSA), v.~Bonomea, 265, Trieste, 34136  Italy}
\address[MIAN]{Steklov Mathematical Institute, Russian Acad. Sciences, Moscow, Russia}
\ead{agrachev@sissa.it}
\author[UIUC]{Yu.~Baryshnikov}
\address[UIUC]{Univ. of Illinois at Urbana-Champaign
1409 W. Green Str., Urbana IL 61801, USA}
\ead{ymb@illinois.edu}
\author[DIMAI]{A.~Sarychev}
\address[DIMAI]{University of Florence, DiMaI, v.~delle Pandette 9, Firenze,  50127 Italy}
\ead{asarychev@unifi.it}


\begin{abstract} We study possibilities to control an ensemble (a parameterized family) of nonlinear control systems  by a  single parameter-independent control. Proceeding by Lie algebraic methods we establish genericity of exact controllability property for finite ensembles, prove sufficient approximate controllability condition for a model problem in $\mathbb{R}^3$, and provide a variant  of  Rashevsky-Chow theorem for  approximate controllability of control-linear ensembles.
\end{abstract}


\end{frontmatter}

\bibliographystyle{elsarticle-harv}
\bibliography{references}

\section{Introduction}

\subsection{Motivation}

    Over the last decade there has been a rise of interest regarding  controllability of ensembles - parameterized families - of nonlinear control systems
\[\dot{x}=f^\theta (x,u), \ \theta \in \Theta ,  \]  by a  single
$\theta$-independent control $u(\cdot)$. Such problems arise for example, from a necessity to control a system  with a "structured uncertainty", when some  parameters of the system are subject to "dispersion".

The problems of designing a control, which  compensates the dispersion, appear for example in NMR spectroscopy.
    Study of the control of Bloch equation under various types of dispersion has been initiated by S.~Li and N.~Khaneja (\cite{LiKh05,LiKh06,LiKh09}).  The state space of Bloch equation is a (matrix) Lie group, and therefore  Lie algebraic notions and tools, such as e.g. Campbell-Hausdorff formula, appeared in its study naturally.
    The core of  the  approach of S.~Li and N.~Khaneja is "generating higher order Lie brackets by use of the control vector fields which carry higher order powers of the dispersion parameters to investigating ensemble controllability".
More  recent publication by K.~Beauchard-J.-M.~Coron-P.~Rouchon (\cite{BCR}), also dedicated to the Bloch equations with dispersed parameter,  invoked {\it analytic} methods to obtain finer results on ensemble controllability.

In the current presentation we search for an extension of the
Lie algebraic approach of geometric control theory onto  ensembles of nonlinear systems.

 {\it Continual} ensemble is an infinite-dimensional control system with  finite-dimensional space of control parameters. Therefore exact controllability would in general fail, a mechanism for such failure is explained in another context in \cite{BMS}. We concentrate  on {\it approximate ensemble controllability}  by means of controls of  {\it fixed} finite dimension.

On the contrast to the  above mentioned publications  we do not use any expansion in the parameter $\theta$, nor do we
assume any smoothness of ensembles in $\theta$. Instead we advocate an approach, which combines use of iterated Lie brackets and hence Taylor series in state variables,  and Fourier-type series in the parameter $\theta$.

We start with {\it finite} ensembles. For such ensembles  the Lie rank criteria of {\it exact} controllability of a single system can be reformulated in  a rather direct way. We  prove in Section~\ref{fecl} that  the property of global controllability  for a {\it finite ensemble of control-linear systems} is generic (Theorem~\ref{lobrens}). In Section~\ref{ferb} we establish (Theorem~\ref{cerb}) global controllability by means of a single scalar control for  a {\it finite ensemble of rigid bodies} with generic inertial parameters.

Two examples of {\it continual ensembles} are studied in Sections~\ref{toym},\ref{rcens}. 

First is a model example of an
ensemble in $\mathbb{R}^3$. We seek for a control, which generates a loop in $\mathbb{R}^2$ and makes the third coordinate  to trace approximately a prescribed target (say, a curve or a surface). Theorem~\ref{lifr}  provides sufficient and necessary condition for the approximate controllability.

In Section \ref{rcens} we study  general ensemble of control-linear systems on a manifold.
     Theorem~\ref{trc}  provides  sufficient  approximate controllability criterion; it is  an ensemble version of Rashevsky-Chow theorem.
     
     Both criteria are formulated  in terms of {\it Lie algebraic span}.  

A number of publications (see \cite{DuSa,Led,SaMa}) contain variants of Rashevsky-Chow theorem in infinite dimension. We explain in Section \ref{rcens} the difference between our criteria and the results of the publications cited.

\subsection{Definitions of ensemble controllability}
Let $M$ be a $C^\ell$-manifold\footnote{In our presentation we consider either analytic case $\ell=\omega$ or infinitely smooth case $\ell=\infty$};    $U \subset \mathbb{R}^r$; $\Theta$ - compact  subset of a Lebesgue measue space.

We consider {\it ensembles} of control systems parameterized by $\theta \in \Theta$
\begin{equation}\label{ens}
\frac{dx^\theta}{dt}=f^\theta(x,u), x^\theta \in M, \ u \in U, \ \theta \in \Theta .
\end{equation}
   Ensemble is {\it finite}, whenever  $\Theta$ is finite and is {\it infinite} otherwise. Note that the control $u(\cdot)$ in \eqref{ens} is assumed to be $\theta$-independent, i.e. {\it all the systems of the ensemble are driven by the same control}.

We are going to study  approximate controllability of ensembles \eqref{ens}.

\begin{defn}[cf. \cite{LiKh05}]
\label{defec}
Let $\alpha(\theta)$ be  an ensemble of initial data
\begin{equation}\label{ensid}
 x^\theta(0)=\alpha(\theta),
\end{equation}
and $\omega(\theta)$ be a target ensemble.

We say that  ensemble \eqref{ens} is $L_p$-approximately steerable
from $\alpha(\theta)$ to $\omega(\theta)$ in  time $T>0$,
     if  for any $\delta >0$ there exists a $\theta$-independent control $\bar{u}(t), \ t \in [0,T]$ (depending  on $\delta$) such that for the  trajectories
of the ensemble
\[\frac{dx^\theta}{dt}=f^\theta(x^\theta,\bar{u}(t))\]
with the initial data (\ref{ensid}),
there holds:
\[\|x^\theta(T)-\omega(\theta)\|_{L_p(\Theta)} < \delta .\]

  Ensemble \eqref{ens} is $L_p$-approximately  controllable (in time $T$) if  for each pair of measurable bounded maps
  $\alpha(\theta), \omega(\theta)$ it is $L_p$-approximately steerable from  $\alpha(\theta)$ to  $\omega(\theta)$ (in time $T$).
   \hfill  $\qed$
\end{defn}


Another definition of controllability, which in  some  cases is slightly stronger, than   approximate controllability,  can  in our view  be useful.

For the  ensemble \eqref{ens} we define a {\it moment} corresponding to a probability density
$p(\theta)$ on the space of parameters $\Theta$:
\[\langle p ,x^\theta(t) \rangle =\int_\Theta \langle p(\theta),x^\theta(t)\rangle d\theta .\]

\begin{defn}
The ensemble  (\ref{ens}) is controllable in momenta if for any finite ensemble of probability densities
 $p_1(\theta), \ldots ,p_m(\theta)$, for any initial data (\ref{ensid}) and each $m$-ple $(\pi_1, \ldots , \pi_m) \in \mathbb{R}^m$
 there exists a control $\bar{u}(t), \ t \in [0,T]$, which steers the ensemble of initial data (\ref{ensid})
to a  terminal ensemble $x^\theta(T)$, for which $\langle p_j ,x^\theta(T) \rangle =\pi_j, \ j=1, \ldots , m$.
\hfill$\qed$
\end{defn}

The criteria of controllability in the momenta can be obtained by the  methods, introduced below.
Still the technicalities differ and we leave the presentation for another occasion.

    Also for continual ensembles we  restrict our attention to the case, where the parameter $\theta$ enters the right hand side of \eqref{ens}, while the initial data does not depend on $\theta: \ \alpha(\theta) \equiv \tilde x.$
An interesting question of controllability of continual ensembles of initial data (interpreted as controllability in the spaces of surfaces/curves) will be treated elsewhere.

\subsection{On  Lie algebraic or geometric control approach}

The {\it geometric
control theory} approaches controllability,
observability and optimality properties
of nonlinear control systems  investigating
  the structure of the Lie algebra, generated by the set of vector fields, which "constitute"  a control  system.
Nagano's theorem puts it in strict terms, stating that two control systems, satisfying  the same {\it Lie relations}, are equivalent up to a coordinate change.
   Identification of  the complete set of Lie relations is in general not possible, but often  a finite subset of this set suffices for establishing controllability (see \cite{AgSch} for details).

    It is rather straightforward  to extend the geometric control approach  to controllability of a single system
    onto the case of {\it finite ensembles}: $|\Theta|=N$. One can just see the   ensemble  as a single system on a carthesian product of $N$ copies\footnote{Minor modification of the approach  allows to deal with control systems defined on different $C^\infty$ manifolds $M^1, \ldots , M^N$.} of the state space $M$  and apply Lie algebraic (Lie rank) methods to establish controllability of this system. Two observations are due: i) for finite ensembles approximate controllability "usually" implies exact controllability; ii) the Lie rank and hence the number of iterated Lie brackets, needed for  the verification of controllability, grows and tends to infinity with $N \to \infty$.

When dealing with infinite and in particular with continual ensembles   tempting is the idea, firt,  to discretize $\Theta$, then to establish (when possible) {\it exact} controllability of the discretized {\it finite} ensemble and finally refine the discretization (increasing the number of "nodes") and   conclude the approximate controllability of the continual ensemble. 

Unfortunately this artless idea seems to fail. The reason is that with the refinement of the discretization the number of the iterated Lie brackets, involved, and hence the
{\it complexity} of the corresponding controls grow unboundedly. The 'nodal' systems are driven by the control of high complexity   to the target, but one looses control of what happens with the systems "between the nodes".

Leaving this idea out we instead view the  ensemble \eqref{ens} as a system in an infinite-dimensional space of functions, defined on $\Theta$,  and
seek for an infinite-dimensional variant of  the {\it method of Lie extensions}. In the next few paragraphs  we
describe the idea informally.

The classical Lie extensions method deals with the vector fields, which are the  sections of
the tangent bundle $TM$. Below we consider instead the fiber bundles  over the base $M$ with the infinite-dimensional fibers $L_p(\Theta,T_xM)$ over each $x \in M$.
 Analogues of vector fields are the sections of the  fiber bundle.
 We  introduce kind of Lie structure for these sections
 by taking  Lie brackets on $M$ for each $\theta \in \Theta$. We  define the Lie extensions and iterating them seek for  an analogue of {\it Lie rank condition}.

 Note that if $\Theta$ is finite then the fiber is just a Carthesian product of a finite number of copies of $T_xM$
 and we come back to the above described approach to finite ensembles.

 Infinite dimensionality intervenes in two ways.
 First, since we take a large but finite number of iterated Lie brackets, we end up with  approximate controllability.
   Second, the usual notions of rank, dimension and linear independence should be treated with more care in the infinite-dimensional situation.
           For model example in Section \ref{toym} we invoke Fourier series in $\theta$; in general it  is useful in constructing appropriate controls.

A natural extension of the notion of ensemble controllability would be the study of controllability in the space of curves or surfaces, or more generally, on the group of diffeomorphisms $\mbox{Diff}\ M $.  A criterion of {\it exact} controllability,
presented in \cite{AgCa}, required the set of controls to be rich enough to allow for
 multiplying vector fields of the system by any smooth functional multiplier.
We look forward to  obtaining results on {\it approximate} controllability on
$\mbox{Diff}\ M $ by means of finite-dimensional control.



\section{Basic assumptions}
\label{prel}

The following two assumptions for the dynamics  of \eqref{ens}
hold for the continual ensembles, treated in
Sections \ref{toym},\ref{rcens}.

Let $M$ be a real analytic ($C^\omega$) manifold.

 \begin{asmp}[Uniform analyticity in $x$]
\label{unal}
The vector fields   $X^\theta(x), \ x \in M$  can be extended  for each $\theta \in \Theta$ to   (complex)-analytic fields $X^\theta(z), \ z \in B_\rho(M)$, where $B_\rho(M)$ is a complex $\rho$-neighborhood of the manifold $M$.  \hfill $\qed$
\end{asmp}

\begin{asmp}[Dependence on parameter $\theta$]
\label{unl2}
The set of parameters $\Theta$ is a separable compact Hausdorff space equipped with a Borel measure.
For each $z \in B_\rho(M)$ the map $\theta \to X^\theta(z)$ is continuous. $ \hfill \qed$
\end{asmp}

\section{Elementary case I: control of a finite ensemble of control-linear systems}
\label{fecl}

For finite ensembles the controllability in momenta and approximate controllability are equivalent\footnote{under full Lie rank condition} to  exact controllability.

We consider finite ensembles  of {\it control-linear} systems and prove that  exact controllability property is {\it generic}.

Finite ensemble of $N$ control-linear systems on a $C^\infty$ manifold $M$
is:
\begin{equation}\label{clin}
\dot{x}^\theta=\sum_{j=1}^rX^{\theta j}(x^\theta)u_j(t), \ x^\theta \in M, \ (u_1, \ldots , u_r) \in \mathbb{R}, \   \theta=1, \ldots ,N.
\end{equation}
Once again the control $u(t)=(u_1(t), \ldots , u_r(t))$ is $\theta$-independent.

For a {\it single system} of
ensemble \eqref{clin}, defined by a  $r$-tuple  ($r \geq 2$) of vector fields $\left(X^{\theta 1}, \ldots , X^{\theta r}\right)$, classical result by C.Lobry \cite{Lob} states,  that global controllability property is generic, i.e.
holds  for each $\left(X^{\theta 1}, \ldots , X^{\theta r}\right)$ from a subset of $\left(\mbox{\rm Vect}^\infty(M)\right)^r$, which is open and dense in $C^\nu$-metric with $\nu$ sufficiently large.

We extend this result to the case of ensembles.

\begin{mthm}
\label{lobrens}
There exists a natural number $\nu$  and a subset  $\mathcal{C} \subset$ \linebreak $ \left(\left(\mbox{\rm Vect}^\infty(M)\right)^N\right)^r,$
which is open and dense in $C^\nu(M)$-metric and such that for each $rN$-tuple of vector fields $\left(X^{\theta j}\right),
\theta =1, \ldots N, \ j=1, \ldots , r$ from $\mathcal{C} $ the ensemble \eqref{clin}
is globally exactly controllable.   \hfill$\qed$
\end{mthm}

\begin{proof}
It suffices to provide a proof for  $r=2$ in \eqref{clin}, or the same, for the ensemble
\begin{equation}\label{2clin}
\dot{x}^\theta=X^{\theta}(x^\theta)u(t)+Y^{\theta}(x^\theta)v(t), \ x^\theta \in M,  \ \theta=1, \ldots ,N.
\end{equation}

We make an obvious step considering on  $M^N$ the cartesian product
of the systems of the ensemble \eqref{2clin}.
Obviously the vector fields   $X=(X^1, \ldots , X^N)$, $Y=(Y^1, \ldots ,Y^N)$,  belong to $\left(\left(\mbox{\rm Vect}^\infty(M)\right)^N\right)^2 \subset  \left(\mbox{\rm Vect}^\infty(M^N)\right)^2$.

The following technical lemma is immediate consequence of Rashevsky-Chow theorem.
\begin{lem}
\label{lrc}
 If the pair $(X,Y)$ is bracket generating at each point of $M^N$, then  ensemble \eqref{2clin} is globally controllable. \hfill$\qed$
\end{lem}

It rests to prove that the bracket generating property is generic in \linebreak $\left(\left(\mbox{\rm Vect}^\infty(M)\right)^N\right)^2$  in $C^\nu$-metric  for some  $\nu$.

C.Lobry's theorem, applied to the couple $(X,Y)$, guarantees  existence and density of  globally controllable couples of vector fields from $\left(\mbox{\rm Vect}^\infty(M^N)\right)^2$,
while we need them to belong to  a smaller set $\left(\left(\mbox{\rm Vect}^\infty(M)\right)^N\right)^2$. Still one can modify the original idea of C.Lobry in order to cover this case.
This  is done in Appendix.
\end{proof}





\section{Elementary case II: Controllability of a finite ensemble of rigid bodies}
\label{ferb}


Consider  an ensemble of $N$  rigid bodies, with the evolution of the momenta, described
by an ensemble of Euler equations
\begin{equation}\label{enrb}
  \dot{K}^\theta=K^\theta \times J^\theta K^\theta + Lu, \ \theta=1, \ldots ,N,\  u \in U \subset \mathbb{R}, \ \mbox{int conv}(U) \ni 0;
\end{equation}
 with  the  {\it scalar} control torque $u(t)$ ({\em in body}),  applied along one and the same direction $L$ to all bodies.

Here $J^\theta \in \mathbf{J}$ are  the (inverses of the) inertia tensors of the bodies;
$\mathbf{J}$ is a closed subset with nonempty interior of the set of  symmetric positive definite
$(3 \times 3)$-matrices.

We restrict ourselves to an open subset  of dynamically asymmetric bodies, or equivalently, the matrices $J^\theta$ with distinct positive eigenvalues in $\mbox{int }\mathbf{J}$.

Without loss of generality we may take as $U=[-\beta,\beta], \ \beta>0$.
Admissible controls $u(t)$ are arbitrary measurable functions, but
piecewise-continuous and  piecewise constant $u(t)$ suffice for controllability.

Consider the cartesian product  of the systems (\ref{enrb})  defined on  $(\mathbb{R}^3)^N$. We provide $(\mathbb{R}^3)^N$ with the euclidean structure and  with the standard volume measure of the cartesian product.

Putting $K=(K^1, \ldots , K^N)$,  $b_L=(L, \ldots , L)$,
$J=\mbox{diag}(J^1, \ldots J^N)$ the $(3N \times 3N)$ block diagonal matrix,
we get a control-affine system on $(\mathbb{R}^3)^N$:
 \[\dot{K}=K \times J K+ b_L u , \]
 (the cross product is applied componentwise).

We denote by $\mathcal{E}_{J}(K)$ the Euler term: $\mathcal{E}_J(K)=K \times J K.$

\subsection{Recollection: controllability result for a single rigid body}

  For \eqref{enrb} with $N=1$    controllability result has been established by two different methods  in \cite{Bon,AgSr0},  see also \cite[Ch.6,8]{AgSch}).

\begin{prop}\label{sirb} For an asymmetric $J^1$ and  generic  $L$,  the single body  is globally controllable. \hfill $\qed$
\end{prop}

Global controllability of single equation \eqref{enrb} can be derived from the bracket generating property
(see Proposition \ref{lobon}), satisfied by the pair of vector fields
$\left(\mathcal{E}_{J^1}(K),b_L(K)\right)$ and the
{\it recurrence property} of  $\mathcal{E}_{J^1}(K)$.

In the next Subsection we  establish controllability of a generic ensemble of $N$ rigid bodies.
Proposition~\ref{sirb} is a special case of this result.

\subsection{Controllability of ensemble of rigid bodies}

\begin{mthm}\label{cerb}
Given $L\in\mathbb{R}^3 \setminus 0$,
and integer $N \geq 1$ there exists an open dense subset $\mathcal{D}  \subset  \mathbf{J}^N$, such that for each $(J^1, \ldots , J^N) \in \mathcal{D}$  the finite ensemble of rigid bodies  \eqref{enrb} is globally exactly controllable by a torque along $L$. Besides for each compact subset   $\mathcal{C} \subset \mathcal{D}$ there exists an upper bound $T_{\mathcal{C}}>0$ for minimal attainability times. \hfill $\qed$
  \end{mthm}

\begin{rem}
The set $\mathbf{J}^N=\mathbf{J} \times \cdots \times \mathbf{J}$ is an open subset of the linear space $\left(Sym (\mathbb{R}^3)\right)^N\ \qed$.
\end{rem}

A more interesting question to be answered is
\begin{questi} Given an $N$-tuple $(J^1, \ldots , J^N)$, with $J^\theta$ pairwise distinct (or belonging to a sphere in $\mathbf J$ and pairwise distinct, or a generic $N$-tuple) and possessing simple eigenvalues, does there exist an open dense subset $\mathbf{L} \subset \mathbb{R}^3$, such that $\forall L \in \mathbf{L}$
the finite ensemble of rigid bodies  \eqref{enrb} is globally controllable? \hfill $\qed$
\end{questi}

We start proving Theorem~\ref{cerb}.

A vector field on a manifold $M$ is {\em recurrent} if $\forall x \in M$, each neighborhood $W_x$ of $x$ and each $t>0$,  there exists a point $\hat x \in W_x$ and time $\hat t >t$ such that $e^{\hat t f}(\hat x) \in W_x$.

Poincare recurrence theorem establishes this property for a broad class of vector fields, which includes $\mathcal{E}_{J^1}(K)$.



 \begin{prop}[Poincare recurrence theorem]\label{Poin}
 If a complete vector field $f$ on a manifold $M$ preserves the volume form (that is, divergence free) and leaves a  set  $A$ of finite volume invariant. Then the restriction of $f$ to $A$ is recurrent.   \hfill $\qed$
 \end{prop}


 It is immediate to see, that the drift vector field $\mathcal{E}_J(K)=K \times J K$
 is divergence free and   preserves
 $\|K\|^2=\sum_{i=1}^r \|K_i\|^2$. Hence the sets  $B_r=\{K| \ \|K\|^2=\sum_{i=1}^r \|K_i\|^2\leq r^2\} $ of finite volume
 are invariant for the volume preserving vector field $\mathcal{E}_J(K)$, wherefrom the recurrence property follows.

The following result states that bracket generating property of a system of vector fields plus the  recurrence property
of one of them suffices for controllability of the respective control-affine system

\begin{prop}[\cite{Bon1}]\label{lobon} Let the  bracket generating condition \linebreak  $\mbox{Lie}(f_0, \ldots , f_r)(x)=T_xM$ hold at each point of $M$, and $f_0$ possess recurrence property on $M$. Then the  control-affine system $\dot{x}=f_0(x)+\sum_{i=1}^r f^i(x)u_i(t)$ (where the control set $U$ is open and contains the origin) is globally controllable on $M$. \hfill $\qed$
\end{prop}


Given the recurrence property of $\mathcal{E}_J(K)$  it remains only to verify the bracket generating property
for the pair of vector fields $\{\mathcal{E}_J(K) , b_L\}$
on  $\left(\mathbb{R}^3\right)^N$.

Given that the vector field $\mathcal{E}_J(K)$ is a polynomial of second degree in $K$ and the vector field $b_L$ is constant,
it is convenient to take into account those iterated Lie brackets, which result in constant vector fields.
These are for example
\begin{eqnarray}\label{libr}
  V^0_{J,L}=b_L(K)=L, V^1_{J,L}=\left[V^0_{J,L},[\mathcal{E}_J,V^0_{J,L}]\right]=2 L \times JL,  \nonumber  \\
   V^{m+1}_{J,L}=\left[V^m_{J,L},[\mathcal{E}_J,V^0_{J,L}]\right], \ m \geq 1.
\end{eqnarray}

 We form $(3N \times 3N)$-matrix
\begin{equation}\label{detr}
R_N\left(J^1, \ldots , J^N;L\right)=\left(
   \begin{array}{cccc}
     V^0_{J^1,L} & V^1_{J^1,L} & \cdots & V^{3N-1}_{J^1,L} \\
     V^0_{J^2,L} & V^1_{J^,L} & \cdots & V^{3N-1}_{J^2,L}\\
     \cdots & \cdots & \cdots & \cdots \\
     V^0_{J^N,L} & V^1_{J^N,L} & \cdots & V^{3N-1}_{J^N,L} \\
   \end{array}
 \right).
\end{equation}

 Bracket generating property for fixed $J$ and $L$
 would be implied by  the non-nullity of the determinant $\det R_N\left(J^1, \ldots , J^N;L\right)$.

For fixed $L$ \eqref{detr} is a polynomial with respect to (the elements of) $J $, its nullity determines an algebraic variety in $\left(Sym (\mathbb{R}^3)\right)^N$. If the determinant  does not vanish {\it identically} on $\left(Sym (\mathbb{R}^3)\right)^N$, i.e. the  variety is proper; then its complement, intersected with $\mathbf{J}^N$, contains an open dense subset of $\mathbf{J}^N$.

\begin{lem}\label{dtr} For each $L$  from an open dense subset of $\mathbb{R}^3$ and for each $N \geq 1$ the determinant  $\det R_N(J^1, \ldots , J^N;L)$ does not vanish
identically on $\left(Sym (\mathbb{R}^3)\right)^N$. \hfill $\qed$
\end{lem}

The proof of the Lemma goes by induction in $N$. We provide the initial inductive step; {\tt the rest of the proof can be found  in  Appendix}.

The set of zeros of the determinant \[R_1\left(J^1;L\right)=\det \left(\left.
        V^0_{J^1,L} \right| V^1_{J^1,L} \left| V^1_{J^1,L}\right.\right)\]
can be  characterized.
 Direct computation shows (see \cite{AgSch}) that for a dynamically asymmetric $J^1$  and for $L$,
 lying in a complement to the  union $\tilde{\mathbf{L}}$ of $3$ straight lines (principal axes) and two (separatrix) planes,
the determinant  $R_1\left(J^1;L\right) \neq 0$.

This fact together with Proposition~\ref{lobon},  implies the statement of  Proposition \ref{sirb}.
\hfill$\qed$


\begin{rem} One should be selective in choosing  iterated Lie brackets,  when establishing bracket generating property.
For example the  constant vector field $\tilde V_{J,L}^2=\left[V^1_{J,L},[\mathcal{E}_J,V^1_{J,L}]\right]$ is  collinear to $V^0_L=L$ for any $J$. \hfill $\qed$
\end{rem}

\subsection{A remark on the bounds for the attainability time}

We will provide the following reinforcement of the previous statement.

\begin{prop}
Let $\mathcal{C} \subset \left(\mathbb{R}^3\right)^N$ be compact. Under the conditions of Theorem \ref{cerb} there exists
uniform upper bound $T_{\mathcal{C}}>0$, such that $\forall \tilde{K},\hat{K} \in \mathcal{C}$ ensemble
\eqref{enrb} can be steered from $\tilde{K}$ to $\hat{K}$ in time $T \leq T_{\mathcal{C}}$. \hfill $\qed$
\end{prop}

\begin{proof}
Existence of minimal attainability $T(\tilde{K},\hat{K})$ time for each couple $\tilde{K},\hat{K}$ is part of classical Filippov theorem. Being the system
globally controllable one can  conclude  that $\hat{K}$ is {\it normally attainable} (\cite{Sus}) from $\tilde{K}$ in
the greater time $T(\tilde{K},\hat{K})+1$. Then each point of a small neighborhood of $\hat{K}$ is attainable from any point in small neighborhood of $\tilde{K}$, or, equivalently, that for any $(x,y)$ in a small neighborhood of $(\tilde{K},\hat{K})$, $y$ can be reached from $x$ in time $T(\tilde{K},\hat{K})+1$. By compactness, one can choose a finite cover by such neighborhoods of $\mathcal{C}\times\mathcal{C}$, implying the proposition.
%
\end{proof}

\section{Continual ensemble of control-linear system: model example}
\label{toym}

\subsection{Problem setting and controllability criterion}

We elaborate our approach to approximate controllability of continual ensembles on a simple model  with  $2$ controls:
\begin{eqnarray}\label{tm1}
    \dot{x}=u, \ \dot{y}=v, \ \dot{z}^\theta=f^\theta(x)v, \\
    x(0)=y(0)=z^\theta(0)=0. \label{tmic}
\end{eqnarray}
The ensemble is  constituted by control-linear systems whose  right-hand side is spanned by the vector fields
\begin{equation}\label{XY}
X=\frac{\partial}{\partial x}, \  Y^\theta=\frac{\partial}{\partial y} + f^\theta(x)\frac{\partial}{\partial z^\theta}, \ \theta \in \Theta .
\end{equation}


One proceeds under Assumptions \ref{unal},\ref{unl2}, in particular $f^\theta(x)$ is analytic in $x$.  We set a slightly modified \\
{\tt Approximate ensemble controllability problem.} {\it Given  $T>0$  and a target function $\hat{z}(p) \in L_\infty(\Theta)$  and $\eps >0$ {\it does there exist  $\theta$-independent controls $u(\cdot), v(\cdot) \in L_2[0,T]$, such that for the trajectory, driven by} $u(\cdot), v(\cdot)$ there holds:}
\begin{equation}
    x(T)=y(T)=0, \
    \int_{\Theta}\|z^\theta(T)-\hat{z}(\theta)\|^2d\theta \leq \eps . \label{zeps}
\end{equation}
Note that, on the contrast to the previous problem setting,  we ask for exact controllability  in coordinates $x,y$.

For simple   model \eqref{tm1}-\eqref{tmic} the trajectory can be computed explicitly:
\begin{eqnarray}
  x(t)=U(t)=\int_0^tu(\tau)d\tau, \ y(t)=V(t)=\int_0^t v(\tau)d\tau, \nonumber \\
 z^\theta(t)=\int_0^tf^\theta(U(\tau))v(\tau)d\tau=\int_0^t f^\theta(U(\tau))dV(\tau).  \label{zexp}
\end{eqnarray}





Consider Taylor expansion for  $f^\theta$ in $x$ at $0$:
\begin{equation}\label{tay}
  f^\theta(x)=\sum_{m=1}^\infty a_m(\theta)x^m, \ a_m(\theta)=\frac{1}{m!}\left.\frac{\partial^m f^\theta}{\partial x^m}\right|_{x=0}.
\end{equation}



The following condition is central for the controllability  of the ensemble \eqref{tm1}-\eqref{tmic}.

\begin{defn}[\it Lie algebraic span condition] The functions $\left(a_m(\theta)\right)_1^\infty$, defined by \eqref{tay},    span dense subspace of $L_2(\Theta)$:
\begin{equation}\label{fuclos}
\overline{span}\{a_m(\theta), \ m=1, \ldots \}=L_2(\Theta). \ \qed
\end{equation}
\end{defn}


\begin{rem} We talk about Lie algebraic condition, since the functions $a_m(\theta)$ are
$z^\theta$-components of  the evaluations  at
 $x=0$ of  the iterated Lie brackets  $\frac{1}{m!}\left((\ad X)^m Y^\theta\right)$ of the vector fields \eqref{XY}.   $\qed$
\end{rem}

\begin{mthm}\label{lifr}
 Ensemble \eqref{tm1} is time-$T$ approximately controllable for each $T>0$
if and only if the Lie algebraic span condition \eqref{fuclos} holds. \hfill $\qed$
\end{mthm}

Rescaling of the time and control $t \to k^{-1}t, \ (u,v) \to (ku,kv), \ k \in \mathbb{R}_+,$ leaves  \eqref{tm1} invariant, therefore we can assume $T=1$.

By  \eqref{zexp}:
\begin{equation}\label{out}
    z^\theta(1)=\int_0^1 f^\theta(U(t))v(t)dt.
\end{equation}

One needs to construct  functions $U(t),v(t)$ such that $U(1)=x(1)=V(1)=y(1)=0$ and $z^\theta(1)$, defined by \eqref{out},  would satisfy the inequality \eqref{zeps} for $T=1$.

To accomplish this we proceed by a variant of moments method.

Assume from now on the  magnitude of the function $U(t)$ to be small, so that the series
\begin{equation}\label{serfu}
  f^\theta(U(t))=\sum_{m=1}^\infty a_m(\theta)(U(t))^m
\end{equation}
will be converging.

We will seek $v(t)$ as a linear combination:
  $v(t)=\sum_{r=1}^R y_r v_r(t)$;
 integer parameter $R$ depends on the rate of approximation and will be specified in a moment.

For the controls defined one derives from the expansion \eqref{serfu}:
\begin{equation}\label{mom1}
z^\theta(1)=\sum_{m=1}^\infty  a_m(\theta)\sum_{r=1}^R\gamma_{mr}y_r.
\end{equation}
where
\begin{equation}
\gamma_{mr}= \int_0^1
\left(U(t)\right)^m v_r(t)dt . \label{dga}
\end{equation}

 If  Lie algebraic span condition is satisfied,
then for each $\eps_1>0$ one can find a finite linear combination  $\sum_{r=1}^R c_m a_m(\theta) $, such that
\begin{equation}\label{apreps1}
 \|\hat z(\theta)-\sum_{r=1}^R c_r a_r(\theta) \|_{L_2(\Theta)} < \eps_1.
\end{equation}

This sets the number $R$, which  depends on the rate of approximation $\eps_1: \ R=R(\eps_1)$.


Our goal is to choose $U(t), v_r(t)$ in such a way that the equation
\begin{equation*}
\sum_{m=1}^\infty  \left(\sum_{r=1}^R\gamma_{mr}y_r\right)a_m(\theta)=\sum_{r=1}^R c_r a_r(\theta)
\end{equation*}
with the coefficients $\gamma_{mr}$, defined by
\eqref{dga}, would be  {\it approximately solvable} with respect to $y_r$.

This fact, {\tt proved in Appendix}, completes the proof of sufficiency part of the Theorem \ref{lifr}.

Now we  prove the {\it necessity}.
If the closure in  \eqref{fuclos} is a proper subspace in $L_2(\Theta)$ take
an element $\nu(\theta)$ orthogonal to the closure: $\int_\Theta \nu(\theta)a_m(\theta)d\theta=0$.
By \eqref{mom1} $\int_\Theta \nu(\theta)z^\theta(1)d\theta=0$ and hence the system can not be  approximately steered to any target function $\hat z(\cdot)$, which is not orthogonal to $\nu (\cdot)$. \hfill $\qed$

\section{Controllability of ensembles of driftless (control-linear) systems. Ensemble version of Rashevsky-Chow theorem}
\label{rcens}

\subsection{Formulation of the result}

Consider the ensemble  of control-linear systems
\begin{equation}\label{clz}
\frac{d}{dt}x^\theta(t)=
\sum_{j=1}^{r}f_j^\theta(x^\theta)u_j(t).
\end{equation}

We study controllability of the ensemble for the case where the parameter $\theta$ enters the dynamics, while the initial data $\tilde x$ and the target $\hat x$ are $\theta$-independent.
Let $d(x,y)$ be a Riemannian distance on $M$.

\begin{defn}\label{acodef}
The ensemble \eqref{clz} is time-$T$
$L_1$-approximately steerable from $\tilde x$ to $\hat x$, if $\forall \varepsilon >0$ there exists a control
$u(\cdot)$, which steers  in time $T$  the ensemble \eqref{clz}  from  $\tilde x$ to $x^\theta(T)$, and:
\begin{equation}\label{fp_contlin}
\int_\Theta d(x^\theta(T),\hat x) d\theta < \eps .\ \qed.
\end{equation}
\end{defn}

\begin{rem}
For technical reasons we opt here for $L_1(\Theta)$-approxi\-mations of the target
on the contrast to  $L_2(\Theta)$-approximations invoked in Section~\ref{toym}. $\qed$
\end{rem}

Let assumptions \ref{unal},\ref{unl2}  hold.


\begin{defn}\label{laspc}
Lie algebraic span condition holds for \eqref{clz},
if  $\forall x \in M$ the evaluations at $x$ of the iterated Lie brackets
of the vector fields $f^\theta_\alpha(x)$
\begin{equation}\label{bralf}
X^\theta_\alpha(x)= [f^\theta_{\alpha_1},[f^\theta_{\alpha_2},[ \ldots ,f^\theta_{\alpha_N}]\ldots ]](x), \ \theta \in \Theta ,
\end{equation}
 span dense subspace
  of the Banach space $L_1(\Theta,T_x M). \  \qed$
\end{defn}

\begin{mthm}[ensemble controllability criterion]\label{trc} Let the  assumptions \ref{unal},\ref{unl2}  and the Lie algebraic span condition (Definition \ref{laspc}) hold  for \eqref{clz}.
      Then for each couple $(\tilde{x},\hat{x})$ and each $T>0$  the ensemble \eqref{clz} is time-$T$ %
      $L_1$-ap-pro\-ximately steerable from $\tilde x$ to $\hat{x}. \  \square$
\end{mthm}

\begin{rem}
For $\Theta$ being finite ($|\Theta|=N$), the  space $L_1(\Theta,T_x M)$ becomes finite-dimensional, isomorphic to $T_xM^N$, and the Lie algebraic span condition is equivalent to the bracket generating condition on $M^N$. In this case stronger  result on {\it exact} ensemble controllability holds (see Section \ref{fecl}).

For $|\Theta|=1$, i.e. for  single system, one gets Rashevsky-Chow theorem.$\qed$
\end{rem}

\begin{rem}
The assumptions, invoked in the formulation, can be weakened. We trust that similar result can be established for the vector fields $f^\theta_{j}$, which are just $C^\infty$-smooth in $x$, as well as the requirement of continuity of $f^\theta_j$  in $\theta$ can be loosened.  $\qed$
\end{rem}

\begin{rem}\label{rash} There was a number of publications (\cite{DuSa,Led,SaMa}) which presented  variants of approximate Rashevsky-Chow theorem in infinite dimension. All those results regard control-linear systems $\dot{y}=\sum_{i=1}^r f^i(y)u_i(t)$,  $y \in E$ in infinite-dimensional vector space $E$ and, roughly speaking, state that whenever approximate bracket generating property holds, i.e. the iterated Lie brackets of the vector fields $g^1, \ldots , g^r$ {\it evaluated at each point of the infinite-dimensional space} span a dense subspace of $E$, then the system is approximately controllable.

     The controlled ensemble  \eqref{clz} can be seen as a control-linear  system in a space $E$ of the functions $x(\theta)=x^\theta$.
One can introduce vector fields on  this space, define the Lie brackets in standard way, and apply the results, just mentioned, to get a controllability criterion.

This criterion would require verification of approximate bracket generating property {\it at each "point" $x(\theta)$ of the functional space}, or the same density of the span of
the iterated Lie brackets \eqref{bralf}, evaluated "along"  each $x(\theta)$.

This means verification of  a vast set of conditions, "indexed" by the elements of a functional space, and is in strong contrast with Theorem \ref{trc}, which just requires verification of bracket generating property at the points of finite-dimensional manifold $M$, or , one could say,  at the constant functions $x(\theta) \equiv \bar x \in M. \ \qed$
\end{rem}

  The rest of the Section is dedicated to the {\it proof of Theorem} \ref{trc}, which is based on  an infinite-dimensional version of the method of Lie extensions.

  According to the method we first establish the possibility to  steer an {\it extended ensemble}
   \begin{equation}\label{ext_ens}
\frac{d}{dt}x^\theta(t)=
\sum_{\alpha \in A}X^\theta_\alpha(x)v_\alpha(t),
\end{equation}
   which involves the vector fields  $X^\theta_\alpha(x)$,  defined by \eqref{bralf}, and  a high-dimensional {\it extended control} $(v_\alpha)$.
   Then we demonstrate how the action of the extended control can be approximated by the action of a small-dimensional original control.
\begin{rem}
Without lack of generality we assume the vector fields $X_j(x)=f_j^\theta(x)$, which define the dynamics of the original ensemble \eqref{clz}, to enter also all the  extended ensembles we invoke. $\Box$
\end{rem}

\subsection{Steering an extended ensemble}

\begin{prop}\label{gloec}
Under the assumptions of the Theorem for each $\tilde x , \hat x \in M$, and each $\eps >0, T>0$
there exists a finite set of multi-indices $A_\eps =\{(\alpha_1, \ldots , \alpha_N)\}$,
and an extended control $\left(v_\alpha(t)\right)_{\alpha \in A},$ which steers in time $T$  the extended ensemble \eqref{ext_ens}
from $\tilde x$ to $x^\theta(T)$, so that \eqref{fp_contlin} holds. $\qed$
\end{prop}

Time and control rescaling $t \to k^{-1}t, \ v_\alpha \to k v_\alpha, \alpha \in A, \ k \in \mathbb{R}_+$  leaves the  control-linear ensemble \eqref{ext_ens}  invariant; therefore whenever controllability is established for some $T_0>0$, it holds for any $T>0$.


Now let us choose any $C^\infty$-smooth vector field $Y(x)$ on $M$ with a trajectory $\bar{x}(t)$, which satisfies the boundary conditions
\[\bar{x}(0)=\tilde{x}, \  \bar{x}(1)=\hat{x} . \]
Denote $\bar{\gamma}=\{\bar{x}(t)| \ t \in [0,1]\} \subset M$.

We prove the following technical Lemma.

\begin{lem}\label{apprY}
Under the assumptions of  Theorem~\ref{trc} there exists a pair of compact neighborhoods $\tilde V, V$ of $\bar \gamma$  $ (\tilde V   \supset \bar V)$ and for each $\eps>0$  a finite set of smooth functions $\left(v_\alpha(x)\right), \ \alpha \in A_\eps$ with supports, contained in $\tilde V$ and such that
\begin{equation}\label{YvX}
 \forall x \in \bar V: \ \left\|Y(x) -\sum_{\alpha \in A_\eps}v_\alpha(x)X^\theta_\alpha(x) \right\|_{L_1(\Theta)}< \eps .
\end{equation}
\end{lem}

\begin{rem}
 The vector  $Y(x)$ in \eqref{YvX} is seen as constant vector-function of $\theta \in \Theta$.
\end{rem}

{\it To prove Lemma~\ref{apprY}} we  fix a compact neighborhood $\bar V$ of $\gamma$ such that at each point $x \in \bar V$  the Lie algebraic span condition \eqref{laspc} is satisfied. Then for each $\eps >0$ and each $x \in \bar V$  there exists a neighborhood $U_x \ni x$ such that
 inequality \eqref{YvX} remains valid for $Y(x')$ in place of $Y(x)$ for each point  $x' \in U_x$.
 We can arrange a finite covering
of $\bar V$ by the neighborhoods $U_i=U_{x_i}, \ i=1, \ldots , N$, in each of which
\[\forall x \in U_i: \ \left\|Y(x) -\sum_{\alpha_i \in A_i}v_{i\alpha_i}(x_i)X^\theta_{\alpha_i}(x) \right\|_{L_1(\Theta)}< \eps , \ i=1, \ldots ,N.  \]

Choose a smooth partition of unity $\{\lambda_i(x)\}$ subject to the covering $\{U_i\}$ of $\bar{V}$. Take $\tilde V$ the union of the supports of $\lambda_i, \ i=1, \ldots , N$. Take $A_\eps=\bigcup_{i=1}^N A_i$ and put for each $\alpha \in A_\eps : \ v_\alpha(x)=\sum_{[i,\alpha_i=\alpha]}\lambda_i(x) v_{i\alpha_i}(x_i). \ \qed$


   Coming {\tt back to the  proof of Proposition}~\ref{gloec} we consider the trajectory $\bar{x}(\cdot)$ of the vector field $Y(x)$, which joins $\tilde{x}$ and $\hat{x}$. Denote $\bar v_\alpha(t)= v_\alpha(\bar x(t)), \ \alpha \in A_\eps$ and consider the time-variant differential equation
\begin{equation*}
\dot{x}^\theta=X_t^\theta(x)=\sum_{\alpha \in A_\eps}\bar v_\alpha(t)X^\theta_\alpha(x)
\end{equation*}
Note that $\bar v_\alpha(t)$ are smooth and the vector field $X_t^\theta(x)$ being analytic in $x$ is locally Lipschitzian.

Let $x^\theta(t)$ be the trajectory of this equation starting at $\tilde x$.

Without lack of generality we may act as if $M$
were a  bounded connected subset of $R^n$.

To  find a bound for $\|x^\theta(T)-\hat x\|$ we compute
\[x^\theta(T)-\hat x=x^\theta(T)-\bar x(T)=
\int_0^T\left(X^\theta
(x^\theta(\tau))-Y(\bar x(\tau))\right)d\tau ,\]
and proceed with the estimates for the norms in $\mathbb{R}^n$.
\begin{eqnarray*}\|x^\theta(t)-\bar x(t)\|=\left\|\int_0^t\left(X_t^\theta(x^\theta(\tau))-Y(\bar x(\tau))\right)d\tau  \right\| \leq \\   \leq \int_0^t\left\|X_t^\theta(x^\theta(\tau))-X_t^\theta(\bar x(\tau))\right\|d\tau + \int_0^t\left\|X_t^\theta(\bar x(\tau))-Y(\bar x(\tau))\right\|d\tau \leq \\ \leq  L_{X}\int_0^t \left\|x^\theta(\tau)-\bar x(\tau)\right\|+\int_0^t\left\|X_t^\theta(\bar x(\tau))-Y(\bar x(\tau))\right\|d\tau d\tau ,\end{eqnarray*}
as long as $x^\theta(\cdot)$ does not leave $\bar V$. Here $L_{X}$ is Lipschitz constant for $X_t^\theta(x)$ on $\bar V$.

Then integrating with respect to $\theta$ and applying  Fubini theorem we get
\begin{eqnarray*}
\int_\Theta
\|x^\theta(t)-\bar x(t)\|d\theta
\leq \int_0^t\int_\Theta\left\|X_t^\theta(\bar x(\tau))-Y(\bar x(\tau))\right\| d\theta d\tau + \\ +   L_{X}\int_0^t \int_\Theta\left\|x^\theta(\tau)-\bar x(\tau)\right\|d\theta d\tau .
\end{eqnarray*}
By virtue of \eqref{YvX}  the last inequality becomes
\[\int_\Theta
\|x^\theta(t)-\bar x(t)\|d\theta \leq \eps t+L_{X}\int_0^t \int_\Theta\left\|x^\theta(\tau)-\bar x(\tau)\right\|d\theta d\tau ,\]
and by virtue of Gronwall lemma
\[\int_\Theta
\|x^\theta(T)-\bar x(T)\|d\theta  \leq \frac{\eps}{L_{X}}\left(e^{L_{X}T}-1\right),\]
wherefrom the claim of Proposition \ref{gloec} follows. $\qed$

\subsection{Lie extension}
We have just proved  approximate controllability for an extended ensemble by means of
a high-dimensional {\it extended} control. Now we have to prove, that the same goal is doable by means of lower-dimensional control. This is done in iterative way via so-called {\it Lie extensions}.

The following result shows, that the control-linear $2$-input ensemble
\begin{equation}\label{eq_6}
  \frac{d}{d t} x^\theta(t) =X^\theta(x) u(t)+Y^\theta(x) v(t),
\end{equation}
and the extended $3$-input ensemble
\begin{equation}\label{eq_7}
  \frac{d}{d t} x^\theta(t) =X^\theta(x) u_e(t)+Y^\theta(x) v_e(t)+[X^\theta , Y^\theta ](x)w_e(t).
\end{equation}
 have (approximately)  the same steering capacities, according to Definition~\ref{acodef}.

\begin{prop}
\label{liex}
If  the   ensemble \eqref{eq_7} can be steered  in time $T$ from $\tilde x$ to $\hat x$ approximately, then the same
is valid  for the   ensemble \eqref{eq_6}. $\qed$
\end{prop}

 Using the statement one can easily complete the proof of Theorem \ref{trc}. Proposition \ref{gloec} demonstrates that an  extended ensemble \eqref{ext_ens} can be steered from $\tilde x$ to $x^\theta(T)$ with \eqref{fp_contlin} satisfied.   By Proposition \ref{liex} the same result can be achieved with a diminished (by $1$) dimension of controls. Proceeding by (inverse) induction we  prove, that the original ensemble  \eqref{clz} can be steered approximately from $\tilde x$ to $\hat x$. \hfill $\qed$

\subsection{Proof of Proposition \ref{liex}}
The construction is based on fast-oscillating functions and on techniques adopted for relaxed controls;
see \cite{AgSr} for an application of these ideas to the control of Navier-Stokes equation.

 Let $u_e(t),v_e(t),w_e(t)$ be the controls, which steer the system \eqref{eq_7} approximately to $\hat x$, acording to Definition~\ref{acodef}. It suffices to establish the statement for smooth $w_e(t)$, as far as smooth functions
  are dense in the space of measurable functions in $L_1$-metric.

 We will use the formula, which is a nonlinear version of the 'variation of constants' method.
 Its more general form - variational formula for time variant vector fields  can be found in \cite[Ch.2]{AgSch}.

Let $\chro_0^t X_\tau d\tau$ denote the flow generated by a time variant vector field $X_t, \ F_0=Id$,
while $e^{tY}$ be the flow, generated by a time-invariant vector field $Y$.

 \begin{lem}\label{varf}
   Let $f_\tau(x),g(x)$ be real analytic in $x$, $f_\tau$ integrable in $\tau$. The flow $P_t=\chro_0^t f_\tau(x)+g(x)u(\tau)d\tau$, corresponding  to
    the differential equation
\begin{equation}\label{caf}
  \dot{x}=f_t(x)+g(x)u(t), \ U(0)=0,
\end{equation}
   can be represented as a composition of two flows
  \begin{equation}\label{vf}
    \chro_0^t \left(f_\tau(x)+g(x)u(\tau)\right)d\tau=\chro_0^t e^{U(\tau)\ad g}f_\tau d\tau \circ e^{gU(t)},    \end{equation}
   where $U(t)=\int_0^t u(\tau)d\tau. \ \qed$
   \end{lem}

 The  operator $\ad_{Z}$,
determined by the vector field $Z$,  acts on vector fields as:
$\ad_Z Z_1=[Z,Z_1]$ -  the Lie bracket of $Z$ and $Z_1$, while the operator exponential
 $e^{\ad_Z}=\sum_{j=0}^\infty \frac{(\ad_Z)^j}{j!}$.

Note that  $e^{U(t) g(x)}$ is time-$U(t)$ element  of the flow of the time-invariant vector field $g$.

To relate the formula \eqref{vf} to fast-oscillating functions we choose a \linebreak $1$-periodic measurable bounded function $v(t)$ with $\int_0^1v(t)dt=0$. Feeding into
\eqref{caf}  a {\it fast-oscillating, possibly high-gain,} control $u_\eps(t)=\eps^{-\alpha} v(t/\eps^\beta)$, $0 \leq \alpha < \beta$, we get by \eqref{vf}
\[\chro_0^t \left(f_\tau(x)+g(x)u_\eps(\tau)\right)d\tau=\chro_0^t e^{\eps^{\beta-\alpha} V(\tau /\eps^\beta)\ad g}f_\tau d\tau \circ e^{\eps^{\beta-\alpha}  V(t/\eps^\beta)}g, \]
where $V(t)=\int_0^t v(\tau)d\tau$  is $1$-periodic Lipschitzian function.

Expanding the exponential $e^{\eps^{\beta-\alpha}  V(\tau /\eps^\beta)\ad g}$ we get
\begin{align*}
  \chro_0^t \left(f_\tau(x)+g(x)u_\eps(\tau)\right)d\tau=&
  \chro_0^t \left( f_\tau(x) +O(\eps^{\beta-\alpha})\right)d\tau \circ \\ \circ \left(I+O(\eps^{\beta-\alpha})\right) =&\chro_0^t f_\tau(x)d\tau \circ (I+O(\eps^{\beta-\alpha})). \nonumber
\end{align*}
This demonstrates that the effect of fast-oscillating perturbation $g(x)u_\eps(\tau)$ tends to $0$ as $\eps \to 0$.

\begin{rem}\label{normsk}
The expression $O(\eps^{\beta-\alpha})$ above regards each  of the seminorms $\|X(x)\|_{s,K}, \|P\|_{s,K}$, which define the convergence of the derivatives of order$\leq s$ on a compact $K \ \Box$.
\end{rem}

\begin{rem}\label{wfo}
  Similar conclusion holds if one takes $u_\eps(t)=w(t)\eps^{-\alpha} v(t/\eps^\beta)$, where $w(\cdot)$
  is, say, Lipschitzian function. The conclusion is achieved by similar reasoning, given the fact that the primitive of      $u_\eps(t)$ in this case is $\eps^{\beta-\alpha}\left(w(t)V(t/\eps^\beta)-\int_0^t V(\tau/\eps^\beta)\dot{w}\tau d\tau \right)=O(\eps^{\beta-\alpha})$, as $\eps \to +0$.
\end{rem}

Coming back to the 2-input system \eqref{eq_6} we choose the controls $u_\eps(t), v_\eps(t)$ of the form
\begin{equation}\label{econ}
  u_\eps(t)=u_e(t)+\eps \dot{U}_\eps(t), \ v_\eps(t)=v_e(t)+\eps^{-1}\hat{v}_\eps(t),
\end{equation}
where $U_\eps(t)$ is function, $U_\eps(0)=0$. Both $U_\eps(t)$ and $\hat v_\eps (t)$ will be specified in a moment.

Feeding the controls  \eqref{econ} into  the  system \eqref{eq_6} we get
\begin{equation}\label{eq_8}
  \frac{d}{d t} x^\theta(t)=X^\theta(x)u_e(t)+Y^\theta(x)\left(v_e(t)+\eps^{-1}\hat{v}_\eps(t)\right)+X^\theta(x) \eps  \dot{U}_\eps(t).
\end{equation}

Applying \eqref{vf} we  represent the flow of \eqref{eq_8} as a
composition of  flows
\begin{align}\label{crex}
\chro_0^t X^\theta(x)u_e(t)+e^{\eps U_\eps(t) \ad X^\theta}Y^\theta(x)(v_e(t)+&\eps^{-1}\hat{v}_\eps(t))dt \circ \nonumber \\ &\circ  e^{\eps U_\eps(t) X^\theta(x)}.
\end{align}


We impose the condition $U_\eps(T)=0$, so that $e^{\eps U_\eps(T) X^\theta(x)}=I$ and we can restrict our attention to the first factor of the composition  \eqref{crex}.

Proceeding with the expansion of the exponential $e^{\eps U_\eps(t) \ad X^\theta(x)} $ in \eqref{crex} we
 rewrite it as
\begin{align}\label{xeps}
 \!\! \chro_0^t(X^\theta (x)u_e(t)+ & Y^\theta (x)v_e(t)+ \\ & + Y^\theta (x) \eps^{-1}\hat{v}_\eps(t)+[X^\theta,Y^\theta](x)U_\eps(t)\hat{v}_\eps(t)+O(\eps))dt. \nonumber
\end{align}

We wish the flow  \eqref{xeps} to approximate the flow generated by the equation  \eqref{eq_7}.
To achieve this we take the functions
\begin{equation}\label{uveps}
  U_\eps(t)=2\sin(t/\eps^2)w_e(t), \ \hat{v}_\eps(t)=\sin(t/\eps^2);
\end{equation}
we choose $\eps$ from the sequence
\begin{equation}\label{neps}
  \eps_n=(T/\pi n)^{1/2}, n=1,2, \ldots ,
\end{equation}
 so that $U_\eps(T)=0$.
Then
\[ U_\eps(t)\hat{v}_\eps(t)=w_e(t)-w_e(t)\cos(2t/\eps^2), \]
so that feeding $U_\eps(t), \hat v_\eps(t)$ into \eqref{xeps} gives us
\begin{eqnarray}\label{xxeps}
  \chro_0^t\!\! \left(X^\theta (x)u_e(t)+ Y^\theta (x)v_e(t)+[X^\theta,Y^\theta](x)w_e(t)+ \right.\\
   +\left.Y^\theta(x) \eps^{-1}\sin(t/\eps^2)-[X^\theta,Y^\theta](x)w_e(t)\cos(2t/\eps^2)+O(\eps)\right)dt. \nonumber
\end{eqnarray}
One can apply formula \eqref{vf} to the flow taking $g=Y^\theta(x)$,   $u(t)=\eps^{-1}\sin(t/\eps^2)$ and denoting $f_t$ the rest of the vector field under the exponential sign. Then we represent the flow \eqref{xxeps} as a composition
\begin{eqnarray*}
  \chro_0^t\!\! \left(X^\theta (x)u_e(t)+ Y^\theta (x)v_e(t)+[X^\theta,Y^\theta](x)w_e(t)- \right. \\
 -\left.[X^\theta,Y^\theta](x)w_e(t)\cos(2t/\eps^2)+O(\eps)\right)dt \circ e^{-\eps \cos (t/\eps^2)}.
\end{eqnarray*}
According to the Remark \ref{wfo} we conclude that
the flow of the equation \eqref{eq_8} can be represented as
\begin{align*}
  \chro_0^t (\ X^\theta (x)u_e(t)+ Y^\theta (x)v_e(t)+  [X^\theta,Y^\theta](x)w_e(t)+ & O(\eps)\ )dt \circ\\  \circ &  (I+ O(\eps))=\\ =\chro_0^t\!\! (\ X^\theta (x)u_e(t)+ Y^\theta (x)v_e(t)+[X^\theta,Y^\theta](x)w_e(t)\ )dt &\circ \left(I+O(\eps)\right).
\end{align*}

Denote by $x^\theta_e(t)$ the trajectory of the $3$-input ensemble \eqref{eq_7}. We  have proved that for the trajectories $x_{\eps_n}(t)$ of the $2$-input ensemble \eqref{eq_6}, driven by the controls  $u_{\eps_n}(t), v_{\eps_n}(t)$, defined by \eqref{econ}-\eqref{uveps}-\eqref{neps},   there holds for each $\theta \in \Theta$:
\begin{equation}\label{eacht}
 \lim_{n \to \infty}\|x^\theta_e(T)-x^\theta_{\eps_n}(T)\|=0.
\end{equation}

 Recall that the real-analytic vector fields $f_j^\theta$ of the ensemble \eqref{clz} depend continuosly on $\theta$; the same holds for the Lie brackets of the vector fields. Then, since $\Theta$ is compact, one can easily check that
 $x^\theta_e(T)$ and $x^\theta_{\eps_n}(T)$ are equibounded for all $\theta$ and $n$£.
 Then by Lebesgue theorem we get from  \eqref{eacht}
\begin{equation*}
 \lim_{n \to \infty}\int_\Theta\|x^\theta_e(T)-x^\theta_{\eps_n}(T)\|d\theta=0. \ \Box
\end{equation*}



\section{Appendix}

\subsection{Proof of Theorem \ref{lobrens}}
 We fix $\dim M=n$.

It is enough to show that the linear ensemble control is generically bracket generating already for two families of vector fields \[X=(X^1,X^2,\ldots, X^N), \ Y=(Y^1,Y^2,\ldots, Y^N).\]

In the trivial ensemble ($N=1$) case, it is easy to see that for a generic pair $X,Y$, the linear span of these fields has rank at least one everywhere on $M$, hence we can always assume that either $X$ or $Y$ are non-vanishing in a vicinity of a point of $M$. After that, the generic generating property follows almost immediately.

If, say $X\neq 0$ at a point $p\in M$, then the codimension of the subset of $(S-1)$-jets of vector fields $Y$ in $M$ at $p$ such that $S$  vectors $\ad_X^k Y(p), \ k=1, \ldots, S-1,$ do {\em not} span $T_pM$, is $S-(n-1)$. Hence,  by R.Thom's transversality theorem, for $S\geq 2n$ this subset is avoided by an open dense set of vector fields $Y$ in $M$.

In the nontrivial  ensemble (finite $|\Theta|>1$) case, a similar approach works, but requires some modifications. We still will choose a control generating a vector field $X$, and will argue that differentiating a vector field corresponding to a different (constant) control iteratively will produce enough vectors to span $TM^N=\bigoplus_\theta TM_\theta$.

There is a small wrinkle here: as the vector fields, tangent to different components $M_\theta$ of the ensemble, do not interact, we will need to have all of $X^\theta(p_\theta)\neq 0$. We cannot however claim that that is true for either $X$ or $Y$: consider, as an example, a point $(x_{1},x_{2},\ldots,x_N)\in M^N$, where $X^{\theta}(x_{\theta})=0$ are non-degenerate zeros. Then for any perturbation of $X^{\theta_i}$'s they will have a point in $M^N$ where all of the vector fields will vanish on one of the ensemble factors.

To overcome this difficulty, we fix a generic collection of controls
$$
\{u_1,\ldots,u_{N+1}\}, u_k=(u_{k,1}, u_{k,2})\in\Real^2,
$$
such that any $2$ of them are linearly independent. Then generically, at least for one of the indices $k$, $L_k^\theta(p_\theta)=u_{k,1}X^\theta(p_\theta)+u_{k,2}Y^\theta(p_\theta)\neq 0$ for all $\theta\in\Theta$. Indeed, otherwise, by pigeonhole principle, there will be at least one of the factors $M^\theta$, two of the vector fields $L_k^\theta, L_l^\theta$ will vanish at $p_\theta \in M^\theta$, and, by assumptions on $u_{\cdot,\cdot}$'s, both $X^\theta$ and $Y^\theta$ would vanish at $p_\theta$. This cannot happen generically, as discussed above.

Hence we can assume that for any point $p=(p_1,\ldots,p_N)\in M^\Theta$, generically, $X^\theta:=L_k^\theta$ is non  null at all factors $p_\theta$ (for a $1\leq k\leq N+1$). Setting $Y^\theta:=L_{(k+1)\mod N}^\theta$, and forming $S$ Lie derivatives \[Y(p), \ad_X Y(p),\ldots, \ad_X^{S-1}Y(p),\] we conclude that generically for $S>2Nn$, these derivatives generate the whole tangent space $T_pM=\bigoplus_\theta T_{p_\theta}M$, at $p$ and in its vicinity. Compactness of $M$ implies that generically, $X,Y$ are bracket generating, and hence the system is fully controllable.
\qed

\subsection{Proof of Lemma \ref{dtr}}

Continuing with the induction on $N$
we introduce a linear map  $\Lambda_{J,L}:\mathbb{R}^3 \to \mathbb{R}^3$:
\begin{equation*}
 K \mapsto [\mathcal{E}_J,V^0_L](K), \hspace{3mm}  [\mathcal{E}_J,V^0_L](K)=L \times JK+K \times JL.
\end{equation*}

Its matrix in the basis of principal axes of the body
has form
\begin{equation}\label{lamd}
\Lambda_{J,L}=D_J \hat L, \
  D_J=\mbox{diag}\left(J_3-J_2,J_1-J_3,J_2-J_1\right), \ \hat L=\left(
            \begin{array}{ccc}
              0 & L_3 & L_2 \\
              L_3 & 0 & L_1 \\
              L_2 & L_1 & 0 \\
            \end{array}
          \right). \nonumber
\end{equation}
Elements of $\Lambda_{J,L}$ are homogeneous of first order in $J$.
The constant vector fields in \eqref{libr}  can be represented as
\[V^{m+1}_{J,L}=\Lambda_{J,L} V^m_{J,L}=\Lambda^m_{J,L}V^1_{J,L}.  \]

By direct computation  $\det \hat L= 2L_1L_2L_3$ and hence given the asymmetry of the body, we conclude thet $\det \Lambda_L \neq 0$, provided that $L$ does not belong to any of coordinate planes $P_i: \ L_i=0; i=1,2,3$.

Let $\mathbf{L}$  be the complement of the union of three principal axes of the body and the three planes $P_1,P_2,P_3$.

According to the aforesaid
\begin{eqnarray}\label{tli}
\forall L \in  \mathbf{L}, \ m \geq 1: \nonumber \\
 \mbox{\it the  vectors} \
V^m_{J,L}, V^{m+1}_{J,L}, V^{m+2}_{J,L} \ \mbox{\it are  linearly independent}.
\end{eqnarray}

The determinant  $R_N$ can be represented as
\footnote{$R_N$ is  (non-commutative) matrix version of  Vandermond determinant. Any explicit computations for such determinant are desirable, but we are not aware of any}
\begin{align*}
   R_N &\left(J^1, \ldots , J^N;L\right)=\\
 &=\det\left(
   \begin{array}{cccc|ccc}
     L & \Lambda_{J^1,L}L &\cdots & \Lambda^{3N-4}_{J^1,L}L & \Lambda^{3N-3}_{J^1,L}L & \Lambda^{3N-2}_{J^1,L}L & \Lambda^{3N-1}_{J^1,L}L\\
     L & \Lambda_{J^2,L}L  & \cdots & \Lambda^{3N-4}_{J^2,L}L & \Lambda^{3N-3}_{J^2,L}L & \Lambda^{3N-2}_{J^2,L}L & \Lambda^{3N-1}_{J^2,L}L\\
     \cdots & \cdots & \cdots & \cdots & \cdots&\cdots & \cdots  \\
     \hline
     L & \Lambda_{J^N,L}L & \cdots & \Lambda^{3N-4}_{J^N,L}L & \Lambda^{3N-3}_{J^N,L}L & \Lambda^{3N-2}_{J^N,L}L & \Lambda^{3N-1}_{J^N,L}L\\
   \end{array}
 \right).
   \end{align*}

We will prove that for $L \in \mathbf{L}$,  the determinant $R_N \left(J^1, \ldots , J^N;L\right)$ defines a nontrivial polynomial in
$J^1, \ldots , J^N$. This is true for $N=1$.

We write $R_N$ in  a  block form
\[R_N=\left(
   \begin{array}{c|c}
     \tilde R_{N-1}& W_{1} \\
\hline
W_{2} & \tilde{R}\\
\end{array}
 \right),
 \]
where $\tilde R_{N-1}$ and $\tilde R$ are $3(N-1) \times 3(N-1)$ and $3 \times 3$ blocks correspondingly, and $W_1,W_2$ have appropriate dimensions.

By induction assumption  $\det \tilde R_{N-1} \neq 0$, while  $\det \tilde{R} \neq  0$ by \eqref{tli}.

To verify that $\det R_N$ does not vanish identically for all $J^1, \ldots , J^N$,
we substitute $\eps J^1, \ldots , \eps J^{N-1}$ in place of $J^1, \ldots , J^{N-1}$.
This results in  multiplication  by $\eps^{k-1}$ of the $k$-th column {\it of the upper block} $\left(R_{N-1}\left|W_1\right.\right), \ k=1, \ldots 3N$.
Denote the resulting matrix by $R_N(\eps)$:
\begin{equation*}
  R_N(\eps)=\left(
   \begin{array}{c|c}
     \tilde R_{N-1}(\eps)& W_{1}(\eps) \\
\hline
W_{2} & \tilde{R}\\
\end{array}
 \right),
\end{equation*}
with
\[\left(R_{N-1}(\eps),W_1(\eps)\right)=\left(\tilde R_{N-1},W_1\right)D^\eps, \ D^\eps=\mbox{diag}\left(1,\eps , \ldots , \eps^{3N-1}\right) .\]

Multiplying the matrix $R_N(\eps)$ from the left by a {\it nonsingular} matrix
\[\left(
   \begin{array}{c|c}
     I & -W_1(\eps) \tilde{R}^{-1} \\
\hline
0 & \tilde{R}^{-1}\\
\end{array}
 \right)
 \]
 we arrive  to the matrix
\[\hat R_N(\eps)=\left(
   \begin{array}{c|c}
     R_{N-1}(\eps) - W_1(\eps)\tilde R^{-1}W_2& 0 \\
\hline
\tilde{R}^{-1} W_{2} & I\\
\end{array}
 \right).
 \]

The elements of $W_1(\eps)$ are $O(\eps^{3N-3})$ as $\eps \to 0$, so are the elements of   $W_1(\eps)\tilde R^{-1}W_2$.

Multiplying the matrix $\hat R_N(\eps)$ from the right by a diagonal matrix $$\mbox{\rm diag}(1,\eps^{-1}, \ldots ,\eps^{-(3N-4)},1,1,1)$$
we get the  matrix
\[\bar R_N(\eps)=\left(
   \begin{array}{c|c}
     R_{N-1} + Y_1(\eps)& 0 \\
\hline
 Y_2(\eps) & I\\
\end{array}
 \right),
 \]
 where $Y_1(\eps)=O(\eps)$. The determinant $\det \bar R_N(\eps)$ is close to $\det \tilde R_{N-1} \neq 0$, and hence differs from $0$, whenever $\eps$ is sufficiently small. Therefore  $R_N(\eps)= R_N \left(\eps J^1, \ldots , \eps J^{n-1}, J^N;L\right)$ is nonsingular for sufficiently small $\eps>0$. \hfill $\qed$

\subsection{Proof of Theorem \ref{lifr}}


Introduce  matrix $\Gamma=(\gamma_{mr}), \ m=1, \ldots , \infty; r=1 ,  \ldots ,R$, with
$\gamma_{mr}$, defined by \eqref{dga}.
Let $\hat \Gamma$ be the upper $(R \times R)$-block of the $(\infty \times r)$-matrix $\Gamma$ and $\tilde \Gamma$ be
the resting infinite block.

We will choose the controls $U(\cdot), v_1(\cdot), \ldots , v_R(\cdot)$ in such a way that the matrix
$\hat \Gamma$ would be (non singular) lower triangular matrix with nonvanishing
diagonal elements. At the same time we will be able to guarantee smallness of   $\|\tilde \Gamma y\|_{\ell_2}$.

We take  Legendre polynomials:
  $P_k(t)=\frac{1}{k!}\frac{d^k}{dt^k}\left((t^2-t)^k\right) $
orthogonal   on $[0,1]$ and put
\begin{equation*}
  U(t)= \int_0^tP_1(s)ds= (t^2-t), \   v_r(t)=P_{2r}(t), \ r=1, \ldots , R.
\end{equation*}
Note that $U(1)=0$ and  $V(1)=\int_0^1 v_r(s)ds=0, \ r=1, \ldots ,R$, since $P_{2r}(t)$ is orthogonal to $1=P_0(t)$.

Evidently  $(U(t))^m$ is polynomial of degree $2m$, hence:
\begin{equation*}
\gamma_{mr}=\int_0^1(U(t))^mv_r(t)dt =0, \ \mbox{for} \ m<r ,
\end{equation*}
while
\[\gamma_{rr}=\int_0^1(U(t))^rv_r(t)dt =\int_0^1 t^{2r}P_{2r}(t)dt\neq 0, \ \forall r .\]

Note that $\gamma_{mr}, m=1, \ldots  , r=1, \ldots $ admit a common  upper bound:
  \begin{align*}
 |\tilde{\gamma}_{mr}|=&\left|\int_0^1 (t-t^2)^m P_{2r}(t)dt  \right| \leq  \\
   \leq& \left(\int_0^1 (t^2-t)^{2m} dt \right)^{1/2}
 \left(\int_0^1 (P_{2r}(t))^2 dt \right)^{1/2}.
 \end{align*}
The first factor is $\leq 2^{-2}$  given that $|t-t^2| \leq 1/4$ on $[0,1]$. The second factor equals $\frac{1}{\sqrt{4r+1}} < 2^{-1}$ for each $r \geq 1$. Hence $|\gamma_{mr}| < 2^{-3}$.

Now we introduce small parameter $\eps$ defining:
\begin{equation*}
  U^\eps(t)=\eps U(t)=\eps (t-t^2), \ v_r^\eps(t)= \eps^{-r} v_r (t)=\eps^{-r}P_{2r}(t).
\end{equation*}

Substituting  $U^\eps(t)$ and $v^\eps_r(t)$ into  \eqref{dga} we get $\gamma^\eps_{mr}=\eps^{m-r}\gamma_{mr}$.

For chosen $U^\eps(t), v^\eps(t)$  the representation \eqref{mom1} for $z^\theta(1)$ takes form
\begin{equation}\label{tmom}
z_\theta(1)=\sum_{r=1}^R a_r(\theta) \gamma_{rr} y_r + \eps \sum_{r=1}^R\sum_{m=r+1}^\infty \eps^{m-(r+1)} a_m(\theta) \gamma_{mr} y_r .
\end{equation}

Let us take $y_r=c_r /\gamma_{rr}, \ r=1, \ldots , R$, where $c_r$ are the coefficients in \eqref{apreps1}, so that \[\left\|\hat z(\theta)-\sum_{r=1}^{R(\eps_1)} a_r(\theta) \gamma_{rr} y_r \right\|_{L_2(\Theta)} < \eps_1.   \]
Obviously the coefficients $y_r$ depend on $\eps_1$.


Now it rests to estimate the second addend at the right-hand side of \eqref{tmom}.
Given that $a_m(\theta)=\left(\int_{C_\rho}  \frac{f^\theta(\zeta)}{\zeta^{m+1}} d\zeta \right)$, we get the estimate
\begin{equation*}
 \left|a_m(\theta)\right|= \left|\int_{C_\rho}  \frac{f^\theta(\zeta)}{\zeta^{m+1}} d\zeta\right| \leq 2\pi \mu_f\rho^{-m},
\ \mu_f=\sup_{(\theta ,\zeta) \in \Theta \times B_\rho}\left|f^\theta (\zeta)\right|.
\end{equation*}

Then
\[\left|\sum_{m=r+1}^\infty \eps^{m-(r+1)} a_m(\theta) \gamma_{mr}\right| \leq 2^{-3}\left|\sum_{s=0}^\infty \eps^{s} \frac{2\pi\mu_f}{\rho^{s+(r+1)}}\right| \leq \frac{\pi\mu_f}{4\rho^R(\rho-\epsilon)},\]
and the second addend in \eqref{tmom} admits an upper bound:
\begin{equation}\label{erud}
  \frac{\eps \pi\mu_f}{4\rho^{R(\eps_1)}(\rho-\epsilon)}\sum_{r=1}^{R(\eps_1)} |y_r|.
  \end{equation}
The term  $\sum_{r=1}^{R(\eps_1)} |y_r|$ admits an upper bound $b_y(\eps_1)>0$.

Now choose $\eps>0$ such that
\[\frac{\eps \pi\mu_f}{2(\rho-\epsilon)}< \frac{\eps_1\rho^{R(\eps_1)}}{b_y(\eps_1)}.\]
Then the estimate \eqref{erud} of the "perturbation term" in \eqref{tmom}  is $< \eps_1$, and
\[\|\hat z(\theta)-z_\theta(1)\|_{L_2(\Theta)} < 2\eps_1. \hfill \qed  \]




\section{Acknowledgements}
YB was supported in part by AFOSR (FA9550-10-1-05678).

\end{document}